\documentclass{amsart}

\usepackage{amssymb,amsmath,color,ytableau,graphicx,float}
\usepackage{amsthm, url}
\usepackage{url}
\usepackage{tikz}
\usepackage{hyperref}
\usepackage[enableskew]{youngtab}
\usepackage{ytableau}

\newtheoremstyle{preview}
  {6pt}                       
  {6pt}                       
  {\itshape}                  
  {}                          
  {\bfseries}                 
  {}                          
  { }                         
  {\thmname{#1}\thmnumber{ #2}\thmnote{ #3}}

\theoremstyle{preview}
\newtheorem*{previewtheorem}{Theorem}

\definecolor{red}{rgb}{1,0,0}
\definecolor{blue}{rgb}{.2,.2,.8}

\makeatletter
\newcommand*{\encircled}[1]{\relax\ifmmode\mathpalette\@encircled@math{#1}\else\@encircled{#1}\fi}
\newcommand*{\@encircled@math}[2]{\@encircled{$\m@th#1#2$}}
\newcommand*{\@encircled}[1]{%
  \tikz[baseline,anchor=base]{\node[draw,circle,outer sep=0pt,inner sep=.2ex] {#1};}}
\makeatother

\def\mex{\textrm{mex}}
\def\a{\textbf{x}}
\def\b{\textbf{y}}
\def\1{\mathbf{1_\ell}}
\def\ks{\mathbf{k_s}}
\def\m{\mathfrak{mi}}

\newtheorem{theorem}{Theorem}[section]

\newtheorem{corollary}[theorem]{Corollary}

\newtheorem*{theorem*}{Theorem}

\theoremstyle{definition}
\newtheorem{definition}{Definition}
\newtheorem{example}{Example}
\newtheorem{remark}{Remark}

\newcommand{\ds}{\displaystyle}

\title[The number of missing integers]{Identities involving the number of missing integers in partitions - combinatorial proofs}
\author[J. Aniceto]{Joselyne Aniceto}\address{Department of Mathematics \\ St. Mary's University \\ San Antonio, TX 78228, USA \\}
 \email{janiceto@stmarytx.edu} 

\author[C. Ballantine]{Cristina Ballantine}\address{Department of Mathematics and Computer Science\\ College of the Holy Cross \\ Worcester, MA 01610, USA \\} 
 \email{cballant@holycross.edu}

\begin{document}
\allowdisplaybreaks

 
\begin{abstract}
    A \textit{missing integer} in a partition $\lambda$ is a positive integer less than the largest part of $\lambda$ that does not appear as a part in $\lambda$. 
    In the recent paper \textit{On the number of missing integers in partitions}, Bhoria, Eyyunni, and Santra examined the number of partitions (overpartitions) of $n$ with a fixed number of missing integers and established generating functions, identities, and congruences for them. In this article, we provide combinatorial proofs for several of their results. 
\end{abstract}

\maketitle

\section{Introduction}\label{intro}
A partition of a positive integer $n$ is a finite non-increasing sequence of positive integers $\lambda=(\lambda_1 , \lambda_2, ..., \lambda_\ell)$ such that $\sum^\ell_{i=1} \lambda_i =n$.
    The integers $\lambda_i$ are called the parts of the partition $\lambda$.
Integer partitions have long played a central role in enumerative combinatorics, with many partition statistics leading to deep arithmetic identities, generating function formulas, and congruence properties.
Among these statistics is the minimal excludant of a partition $\lambda$ (or $\mex(\lambda)$)  \cite{mex}, defined as the smallest positive integer that does not appear as a part of $\lambda$.  

In \textit{On the number of missing integers in partitions} \cite{missing}, Bhoria, Eyyunni and Santra focus on the entire set of missing positive integers less than the largest part of a partition  rather than the minimal missing part. They consider  the number $P(n,m)$ of partitions of $n$ with exactly $m$ missing integers less than the largest part as well as the number $\overline P(n,m)$ of overpartitions of $n$ with the same property. (An overpartition is a partition in which the first occurrence of a part may be overlined.)
They derived several generating functions, identities  and congruences involving these (over)partition functions. In particular, they prove  \cite[Corollaries 2.2 and 3.2]{missing}
\begin{align}\label{Pgen} \sum_{n=0}^\infty\sum_{m=0}^\infty P(n,m)w^mq^n& = \frac{((w-1)q;q)_\infty}{(wq;q)_\infty}\\ \label{oPgen}\sum_{n=0}^\infty\sum_{m=0}^\infty \overline P(n,m)w^mq^n& = \frac{((w-2)q;q)_\infty}{(wq;q)_\infty}, 
    \end{align}
where $(a;q)_{\infty}$ is the usual Pochhammer symbol \begin{equation*}
    (a;q)_{\infty} := \prod^{\infty}_{j=1} (1-aq^{j-1}), \ \ a,q \in \mathbb C, |q|<1.
\end{equation*}

In the present paper, we find combinatorial interpretations for each of the  products on the right hand side of \eqref{Pgen} and \eqref{oPgen} and give bijective proofs of the obtained identities thus providing combinatorial proofs of \cite[Corollaries 2.2 and 3.2]{missing}.
 Then, we  simplify the bijections in order to give combinatorial proofs for special cases thus proving \cite[Corollaries 2.3, 2.4 and 2.5]{missing} as well as \cite[Corollaries 3.3, 3.4 and 3.5]{missing}. 

In  \cite{missing}, the authors  ask for a bijective proof of the following result which they prove analytically. 
\begin{previewtheorem}[3.5]\cite[Theorem 2.9]{missing} Let $k$ be a positive integer. The number of distinct integers less than or equal to the largest part appearing less than $k$ times in all partitions of $n$ equals the number of parts different from $k$ in all partitions of $n$.
    \end{previewtheorem}
    We provide a bijective proof of this theorem.

The remainder of this paper is organized as follows.
In Section \ref{notation}, we establish the notation and preliminary definitions used throughout. 
Section \ref{proofs} is devoted to combinatorial proofs of results from \textit{On the number of missing integers in partitions} \cite{missing}. 
Finally, in Section \ref{conclusion}, we offer some concluding remarks.

\section{Definitions and Notation}\label{notation}

In this section we introduce the notation necessary for the proofs outlined in Section \ref{proofs}.

 As noted in Section \ref{intro},   a \textit{partition} of a positive integer $n$ is a finite non-increasing sequence of positive integers $\lambda=(\lambda_1 , \lambda_2, ..., \lambda_k)$ such that $\sum^k_{i=1} \lambda_i =n$.
    We refer to the integers $\lambda_i$ as the \textit{parts} of $\lambda$.
    The sum of the parts of $\lambda$ is called  the \textit{size} of $\lambda$, denoted  by $|\lambda|$. 
We write $\lambda\vdash n$ to mean that $\lambda$ is a partition of $n$. By convention, the only partition of $0$ is the emptyset. 

We identify a partition with the multiset of its parts and use multiset operations throughout. 
For an integer $a$, we write $a\in \lambda$ if $a$ is a part of $\lambda$, and denote by $m_{\lambda}(a)$  its \textit{multiplicity} in $\lambda$, that is the number of times $a$ occurs as a part in $\lambda$.
The number of parts of $\lambda$, denoted by $\ell (\lambda)$ is called the \textit{length} of $\lambda$.

Let $\mathcal P$ denote the set of all unrestricted partitions and $\mathcal D$ denote the set of all partitions with distinct parts.  For more on the theory of partitions, we refer the reader to \cite{andrews}.

 Given a partition $\lambda=(\lambda_1, \lambda_2, \ldots, \lambda_k)$, the 
    Ferrers diagram of $\lambda$ is a left aligned array of boxes such that the $i$-th row from the top has $\lambda_i$ boxes.  

For example, the Ferrers diagram of $\lambda=(4,4,2,1,1)$ is 

$$\small\ydiagram{4,4,2,1,1}$$ \medskip

We refer to the box in row $i$ and column $j$ as the box in position $(i,j)$. We write $(i,j)\in \lambda$ if there is a box in position $(i,j)$ in the Ferrers diagram of $\lambda$.

    The conjugate $\lambda'$ of a partition $\lambda$ is the partition whose Ferrers diagram is obtained from the Ferrers diagram of $\lambda$  by exchanging rows and columns, i.e., if  $(i,j)\in \lambda$ , then  $(j,i)\in\lambda'$.

To simplify the exposition, $\lambda$ can refer to a partition, its Ferrers diagram, or its multiset of parts. 

Overpartitions were first introduced in \cite{overpartitions}.
An \textit{overpartition} of $n$ is a non-increasing sequence of positive integers  whose sum is $n$ in which the first occurrence of a part may be overlined. For example, the overpartitions of $n=3$ are $$(3), (\overline 3), (2,1), (\overline 2, 1), (2, \overline 1), (\overline 2,\overline 1), (1,1,1), (\overline 1, 1,1).$$
If $\lambda$ is an overpartition, we denote by $\overline \ell(\lambda)$  the number of overlined parts in $\lambda$ and by  $\ell(\lambda)$  the total number of parts in $\lambda$, i.e., the number of overlined and non-overlined parts.
In this paper, a distinct overpartition is an overpartition with distinct parts in which no part may occur both overlined and non-overlined. 
Let $\overline{\mathcal D}$ denote the set of all distinct overpartitions.

We will also make use of \textit{doubly overlined  overpartitions}, in which overlined parts may carry either a single overline or a double overline. 
A doubly overlined distinct overpartition $\lambda$ is a distinct overpartition in which each overlined part is assigned either a single overline or double overline. 
As with distinct overpartitions, each part size occurs at most once. 
Thus, for each positive integer, at most one of $i$, $\overline{i}$, or $\overline{\overline{i}}$ may appear in $\lambda$. 
We denote the set of all doubly overlined distinct overpartitions by $\overline{\overline{\mathcal D}}$.
For example, $(\overline{\overline 5}, 4, \overline 2)$, $(\overline 5, 4, \overline 2)$, $(5, 4,  2)$ are elements of $\overline{\overline{\mathcal D}}$ and are doubly overlined distinct partitions of $11$, whereas $(\overline{\overline 7}, \overline 2, 2)\not\in \overline{\overline{\mathcal D}}$, since there are two  parts of size $2$. 

For $\lambda\in \overline{\overline{\mathcal D}}$, we denote by $\overline\ell(\lambda)$  the number of singly overlined parts in $\lambda$, and by $\overline{\overline\ell}(\lambda)$  the number of doubly overlined parts in $\lambda$. 
As in the case of overpartitions, $\ell(\lambda)$ denotes the total number of parts of $\lambda$.

Finally, if $\lambda\in {\overline{\mathcal D}}$ or $\overline{\overline{\mathcal D}}$, we write $\mathbf i\in \lambda$ to indicate that $\lambda$ contains a part of size $i$, regardless of its overline status; that is, one of $i$, $\overline{i}$, or $\overline{\overline{i}}$ belongs to $\lambda$. 
Likewise, $\mathbf i\not\in \lambda$ means that $\lambda$ contains no part of size $i$, i.e. neither of $i, \overline i, \overline{\overline i}$ is in $\lambda$.

If $(\lambda, \mu)$ is a pair of partitions, overpartitions, or doubly overlined partitions, we write $(\lambda, \mu)\vdash n$  whenever $|\lambda|+|\mu|=n$. 

If $\lambda$ is a partition or an overpartition, the \textit{missing integers} of $\lambda$ are the positive integers less than the largest part of $\lambda$ that do not occur as parts of $\lambda$. Given a partition or an overpartition $\lambda$, we denote by 
$\m(\lambda)$ denote the number of missing integers of $\lambda$.
For example, if  $\lambda=(5,5,3,1)$, then the missing integers are $2$ and $4$, so $\m(\lambda)=2$. 
Similarly, if $\mu=(5,5,\overline 3, 1)$ is an overpartition, the missing integers are also $2$ and $4$, and hence $\m(\mu)=2$. 
For $n,m\geq 0$, we define
$$\mathcal{P}(n,m):=\{\lambda \vdash n \mid \m(\lambda)=m \}.$$ Thus $P(n,m)= |\mathcal{P}(n,m)|$.

The Ferrers diagram of an overpartition $\eta$ is obtained by first drawing the Ferrers diagram of the partition obtained by ignoring the overlines in $\eta$. 
Then, for each overlined part $\overline i$, we shade the rightmost of box in the last row of size $i$.

The conjugate of an overpartition $\eta$ is the overpartition $\eta'$ whose Ferrers diagram is obtained from the Ferrers diagram of $\eta$ by  interchanging rows and columns. If there is a shaded box at the end of a row in $\eta$, there is a shaded box at the end of the corresponding column in $\eta'$. Since shaded boxes only occur in corners, the obtained diagram is the Ferrers diagram of an overpartition. 

\begin{example}
    If $\eta=(\overline 4, 4, 4, \overline 3, 3 ,3, 1,1)$, then $\eta'=(8, \overline 6, 6, \overline 3)$. 
    The Ferrers diagrams are shown below. 
    
    $$\small\ytableausetup{nosmalltableaux}
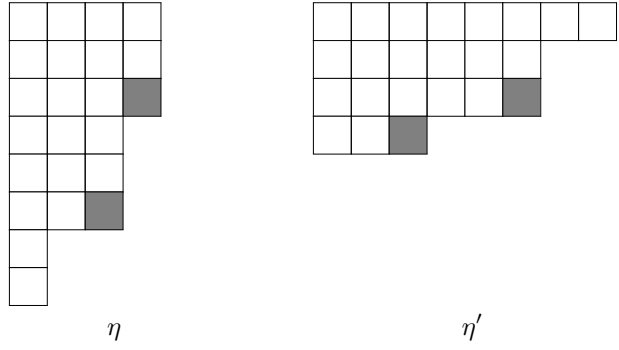
\begin{ytableau}*(white) & *(white) & *(white) &*(white) \\*(white) & *(white) & *(white) &*(white)\\ *(white) & *(white) & *(white) &*(gray)\\ *(white) & *(white) & *(white)\\*(white) & *(white) & *(white)\\ *(white) & *(white)& *(gray)\\ *(white)\\  *(white)
\end{ytableau}\hspace{2cm}
\small\ytableausetup{nosmalltableaux}\begin{ytableau}*(white) & *(white) & *(white) &*(white)& *(white) & *(white) & *(white) &*(white) \\ *(white) & *(white) & *(white) &*(white) & *(white) &*(white)\\ *(white) & *(white) & *(white)& *(white) & *(white) & *(gray)\\ *(white)& *(white) & *(gray)
\end{ytableau}$$ 
$$\eta \hspace{4.5cm} \eta'\ \ \ \ $$
\end{example}

\section{Combinatorial proofs of results in \cite{missing}} \label{proofs}

Throughout the remainder of the paper, we represent the missing integers of a partition or an overpartition $\lambda$ by marking its Ferrers diagram as follows.
 A positive integer $j$ is missing from $\lambda$ if and only there exists $1\leq i\leq \ell(\lambda)$ such that  $(i,j), (i, j+1)\in \lambda $  but  $(i+1, j)\not \in \lambda$. 
We visualize this by placing a $*$ in the bottom cell of every column that is not the rightmost column of its length. 
If  $*$ appears in the position $(i,j)$ in the marked Ferrers diagram of $\lambda$, then $\lambda$ has no part of size $j$. 
\begin{example}
  The marked Ferrers diagram of $\lambda=(9, 7, 6, 6, 3, 2)\in \mathcal P(35, 4)$ is 

$$\young(\hfil\hfil\hfil\hfil\hfil\hfil\hfil*\hfil,\hfil\hfil\hfil\hfil\hfil\hfil\hfil,\hfil\hfil\hfil\hfil\hfil\hfil,\hfil\hfil\hfil**\hfil,\hfil\hfil\hfil,*\hfil)$$ 
There is a $*$ in each of the following positions: $(1, 8), (4, 4), (4, 5), (6,1)$. 
The missing numbers in $\lambda$ are $8, 4, 5, 1$.
\end{example} 

We begin by giving a combinatorial proof of \cite[Corollary 2.2]{missing}. 
We then show that \cite[Corollaries 2.3 and 2.4]{missing} follow immediately by conjugation.

\begin{theorem}\label{cor2.2}\cite[Corollary 2.2]{missing}
    We have \begin{equation}\label{eq:cor2.2} \sum_{n=0}^\infty\sum_{m=0}^\infty P(n,m)w^mq^n= \frac{((w-1)q;q)_\infty}{(wq;q)_\infty}. 
    \end{equation}
\end{theorem}

\begin{proof} We  first interpret the right hand side of \eqref{eq:cor2.2} combinatorially. 
Let $$\mathcal A(n):=\{(\lambda, \mu)\vdash n \mid  \lambda\in \mathcal D, \,  \mu\in \mathcal P\}.$$  
The coefficient of $q^n$ in the  $q$-series expansion of the right hand side of \eqref{eq:cor2.2}
is given by $$\sum_{(\lambda, \mu) \in \mathcal A(n)}\sum_{k=0}^{\ell(\lambda)} (-1)^k\binom{\ell(\lambda)}{k}w^{k+\ell(\mu)}.$$   
Let $$\overline{\mathcal A}(n):=\{(\lambda, \mu)\vdash n \mid  \lambda\in \overline{\mathcal D}, \,  \mu\in \mathcal P\}.$$
Note that if $(\lambda, \mu)\in \overline{\mathcal A}(n)$, then $\lambda$ can have any number of overlined parts, i.e., $0\leq \overline\ell(\lambda)\leq \ell(\lambda)$. 
Let $$\overline{\mathcal A}(n,m):=\{(\lambda, \mu)\in \overline{\mathcal A}(n)\mid \overline\ell(\lambda)=m-\ell(\mu)\}.$$ 
Then, the coefficient of $w^mq^n$ in  the  series expansion of the right hand side of \eqref{eq:cor2.2} is 
$$\#\{(\lambda, \mu)\in \overline{\mathcal A}(n,m)\mid \overline\ell(\lambda) \text{ even}\}-\#\{(\lambda, \mu)\in \overline{\mathcal A}(n,m)\mid \overline\ell(\lambda) \text{ odd}\}. $$

Let $$\overline{\mathcal E}(n,m):=\{(\lambda, \mu)\in \overline{\mathcal A}(n,m)\mid \ell(\mu)=m \text{ and if }  i\in \mu \text{ then } \mathbf{i}  \in \lambda \}.$$

Note that if $(\lambda,\mu)\in \overline{\mathcal E}(n,m)$, then $\overline\ell(\lambda)=0$.

Next, we define an involution $\varphi_m$ on $\overline{\mathcal A}(n,m)\setminus \overline{\mathcal E}(n,m)$ that changes that parity of  $\overline\ell(\lambda)$.
Let  $(\lambda,\mu)\in \overline{\mathcal A}(n,m)\setminus \overline{\mathcal E}(n,m)$. 

If $\ell(\mu)=m$, then $\overline\ell(\lambda)=0$ and there exists a part $i\in \mu$ such that $ i\not\in \lambda$. Let $t:=\max\{i\in \mu \mid  i\not\in \lambda\}$ and define $\varphi_m(\lambda, \mu):=(\lambda\cup (\overline t), \mu\setminus(t)).$

If $\ell(\mu)<m$, then $\overline\ell(\lambda)\geq 1$. 
Let $\overline u$ be the largest overlined part in $\lambda$. If there is $i\in \mu$, $i>u$, $i\not \in \lambda$, let $t:=\max\{i\in \mu \mid i>u,\,  i\not\in \lambda\}$ and define $\varphi_m(\lambda, \mu):=(\lambda\cup (\overline t), \mu\setminus(t)).$ 
Else, define $\varphi_m(\lambda, \mu):=(\lambda\setminus(\overline u), \mu\cup (u)).$

Since $\varphi_m$ is an involution on $\overline{\mathcal A}(n,m)\setminus \overline{\mathcal E}(n,m)$ that changes the parity of $\overline\ell(\lambda)$, the coefficient of $w^mq^n$ in  the  series expansion of the right hand side of \eqref{eq:cor2.2} equals $|\overline{\mathcal E}(n,m)|$.
We complete the proof by  establishing a bijection $$\psi_m:\overline{\mathcal E}(n,m)\to \mathcal P(n,m).$$ Start with $(\lambda, \mu)\in \overline{\mathcal E}(n,m)$. 
We create the partition $\lambda\cup \mu$ and decorate its Ferrers diagram as follows: for each part $j\in \mu$ we mark  the last box of the first $m_\mu(j)$ rows of length $j$  with $\bullet$. 
Since each part of $\mu$ is a part of $\lambda$, the last row of each fixed size has the last box unmarked. 
We define $\psi_m(\lambda, \mu):=(\lambda\cup \mu)'$. 
To see that $(\lambda\cup \mu)'\in \mathcal P(n,m)$, when conjugating the decorated Ferrers diagram of $\lambda\cup \mu$,  if  the last box of a row is marked with $\bullet$, we mark the last box of the corresponding column with $*$.    
If there is a $*$ in position $(i,j)$ in the $*$-decorated Ferrers diagram for $(\lambda\cup \mu)'$, then $j\not\in (\lambda\cup \mu)'$. 
Hence, there are exactly $\ell(\mu)=m$ missing integers in $(\lambda\cup \mu)'$ and thus $(\lambda\cup \mu)'\in \mathcal P(n,m)$.\end{proof}

\begin{example} Let $(\lambda, \mu)=((6,5,4,2,1),(6,4,4,1))\in \overline{\mathcal E}(33,4) $. 
The decorated Ferrers diagrams for $\lambda\cup \mu$ and $(\lambda\cup \mu)'$ are shown below.  
The missing integers in $(\lambda\cup \mu)'$ are $1,4,5,8$. Thus $(\lambda\cup \mu)'\in \mathcal P(33,4)$.

$$\young(\hfil\hfil\hfil\hfil\hfil\bullet,\hfil\hfil\hfil\hfil\hfil\hfil,\hfil\hfil\hfil\hfil\hfil,\hfil\hfil\hfil\bullet,\hfil\hfil\hfil\bullet,\hfil\hfil\hfil\hfil,\hfil\hfil,\bullet,\hfil) \hspace{2cm} \young(\hfil\hfil\hfil\hfil\hfil\hfil\hfil*\hfil,\hfil\hfil\hfil\hfil\hfil\hfil\hfil,\hfil\hfil\hfil\hfil\hfil\hfil,\hfil\hfil\hfil**\hfil,\hfil\hfil\hfil,*\hfil)
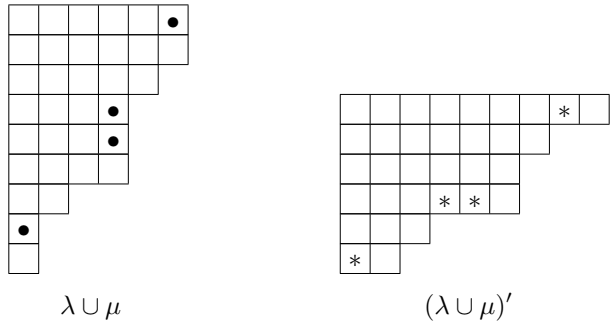
$$  
$$\lambda\cup \mu \hspace{4cm} (\lambda\cup \mu)'\ \ \ \ \ $$
\end{example}

\begin{remark} We note that the proof of 
    Theorem \ref{cor2.2} provides a combinatorial argument for the following identity. For $n,m\geq 0$,
    $$P(n,m) = \#\{(\lambda, \mu)\in \overline{\mathcal A}(n,m)\mid \overline\ell(\lambda) \text{ even}\}-\#\{(\lambda, \mu)\in \overline{\mathcal A}(n,m)\mid \overline\ell(\lambda) \text{ odd}\}. $$
\end{remark}

\begin{remark}
   In Theorem \ref{cor2.2}, when $m=0$,
    $\overline{\mathcal{A}}(n,0)=\{ (\lambda,\emptyset)\vdash n \mid \lambda\in \mathcal{D} \}= \overline{\mathcal{E}}(n,0)$. 
    Then $\varphi_0$ is an involution on the empty set, and $\psi_0$ is conjugation, i.e., $\psi_0(\lambda, \emptyset)=\lambda'$. 
    It is known that the conjugate of a distinct partition is a partition with consecutive parts, i.e., a gap free partition. Thus, conjugation gives a bijective proof of \cite[Corollary 2.3]{missing}, where the number of gap-free partitions of $n$ is denoted by $\mathcal{P}^*(n)$.
\end{remark}
\begin{corollary} \cite[Corollary 2.3]{missing}
The number of partitions of $n$ in which no integer is missing, $P(n,0)= \mathcal{P}^*(n)$, equals the number of distinct partitions of $n$.
We have
\begin{equation*}
\displaystyle\sum^{\infty}_{n=0} \mathcal{P}^*(n) q^n = (-q;q)_{\infty}.
\end{equation*}
\end{corollary}

\begin{remark}
    In Theorem \ref{cor2.2}, when $m=1$, 
   \begin{align*}\overline{\mathcal{A}}(n,1)& =\{(\lambda,\mu )\vdash n \mid \lambda \in \mathcal{D} , \, \ell(\mu)=1 \} \cup \{ (\lambda, \emptyset) \mid \lambda\in \overline{\mathcal{D}}, \, \overline{\ell}(\lambda)=1\}\\
\overline{\mathcal{E}}(n,1)& =\{(\lambda,\mu)\vdash n  \mid \lambda \in \mathcal D,\, \mu=(i), \,   i\in \lambda \}. \end{align*}

Thus, if $(\lambda, \mu)\in \overline{\mathcal{E}}(n,1)$, then $\lambda\cup \mu$ is a partition with exactly one part repeating twice and all other parts distinct. Hence, $\psi_1^{-1}$, which is conjugation, gives a bijective proof of \cite[Corollary 2.4]{missing}.
\end{remark}

\begin{corollary} \cite[Corollary 2.4]{missing}
The number of partitions of $n$ in which exactly one integer is missing, $P(n,1)$, is equal to the number of partitions of $n$ in which exactly one part repeats twice while all other parts occur onlly once. 
We have 
\begin{equation*}
    \displaystyle\sum^{\infty}_{n=0} P(n,1) q^n = (-q;q)_{\infty} \displaystyle\sum^{\infty}_{k=1} \frac{q^{2k}}{1+q^k}.
\end{equation*}
\end{corollary}

\begin{example} Let  $\lambda=(5,4,4,4,2,1,1)\in \mathcal P(21,1)$. The only missing integer in $\lambda$ is $3$.  Then $\psi_1^{-1}(\lambda)=\lambda'=(7,5, 4,4,1)$, the only repeated part in $\lambda'$ is $4$ and $m_{\lambda'}(4)=2$.

$$\young(\hfil\hfil\hfil\hfil\hfil,\hfil\hfil\hfil\hfil,\hfil\hfil\hfil\hfil,\hfil\hfil*\hfil,\hfil\hfil,\hfil,\hfil)\hspace{2cm} \young(\hfil\hfil\hfil\hfil\hfil\hfil\hfil,\hfil\hfil\hfil\hfil\hfil,\hfil\hfil\hfil\bullet,\hfil\hfil\hfil\hfil,\hfil) 
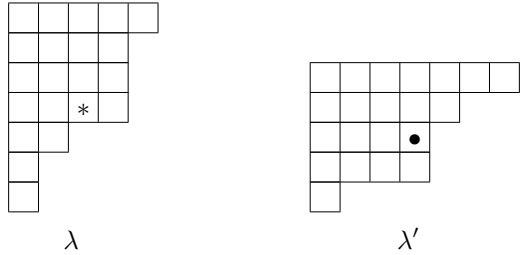
$$
$$\lambda\hspace{4.2cm} \lambda'\ \ \ \ \ $$
\end{example}

\begin{remark}
    The involution $\varphi_m$ on $\overline{\mathcal A}(n,m)\setminus \overline{\mathcal E}(n,m)$ defined in the proof of Theorem \ref{cor2.2} applied to $(\lambda, \mu)$ also changes the parity of $\ell (\mu)$.
\end{remark}

\begin{definition} Similar to the notation in  \cite{missing}, we define the sets \begin{align*}\mathcal M_e(n)& :=\{\lambda\vdash n \mid \m(\lambda) \text{ even}\}=\bigcup_{k\geq 0}\mathcal P(n,2k)\\ \mathcal M_o(n)& :=\{\lambda\vdash n \mid \m(\lambda) \text{ odd}\}=\bigcup_{k\geq 1}\mathcal P(n,2k-1),\end{align*} and set $M_e(n):=|\mathcal M_{e}(n)|$ and $M_o(n):=|\mathcal M_{o}(n)|$.
    The functions ${M}_e(n)$ and ${M}_o(n)$ denote the number of partitions of $n$ with an even/odd number of missing integers, respectively. 
\end{definition}

The combinatorial proof of Theorem \ref{cor2.2} can be adapted to provide a combinatorial proof of \cite[Corollary 2.5]{missing}.

\begin{corollary}\label{cor2.5}\cite[Corollary 2.5]{missing}
    We see that \begin{equation}\label{eq:cor2.5} \sum_{n=0}^\infty\sum_{m=0}^\infty P(n,m)(-1)^mq^n= \sum_{n=0}^\infty \left( M_e(n)- M_o(n)\right) q^n=\frac{(-2q;q)_\infty}{(-q;q)_\infty}. \end{equation}
\end{corollary}

\begin{proof}

    As noted in \cite{missing}, the infinite product $(-2q;q)_\infty$ is the generating function for the number of distinct overpartitions of $n$.  Let \begin{align*}
        \overline{\mathcal A}_e(n)& :=\{(\lambda, \mu)\in \overline{\mathcal A}(n)\mid \ell(\mu) \text{ even}\}\\ \overline{\mathcal A}_o(n)& :=\{(\lambda, \mu)\in \overline{\mathcal A}(n)\mid \ell(\mu) \text{ odd}\}
    \end{align*}
Then, $\displaystyle\frac{(-2q;q)_\infty}{(-q;q)_\infty}$ is the generating function for the sequence $\{|\overline{\mathcal A}_e(n)|-|\overline{\mathcal A}_o(n)|\}$. We have $\overline{\mathcal A}(n)=\ds\bigcup_{m\geq 0}\overline{\mathcal A}(n.m)$ and we define 
 $$\overline{\mathcal E}(n):=\bigcup_{m\geq 0}\overline{\mathcal E}(n,m)=\{(\lambda, \mu \in \overline{\mathcal A}(n) \mid \overline \ell(\lambda)=0 \text{ and if }i\in \mu \text{ then }i \in \lambda\}.$$ From the definitions, it follows that  $\overline{\mathcal A}(n)\setminus \overline{\mathcal E}(n)=\ds \bigcup_{m\geq 0}\overline{\mathcal A}(n,m)\setminus \overline{\mathcal E}(n,m)$. We define an involution $\phi$ on $\overline{\mathcal A}(n)\setminus \overline{\mathcal E}(n)$ as follows: if $(\lambda,\mu)\in \overline{\mathcal A}(n)\setminus \overline{\mathcal E}(n)$ and $\ell(\mu)+\overline \ell(\lambda)=m$, then $(\lambda, \mu)\in \overline{\mathcal A}(n,m)\setminus \overline{\mathcal E}(n,m)$ and we define $\phi(\lambda,\mu):=\varphi_m(\lambda,\mu)$. Since $\varphi_m$ changes the parity of $\ell(\mu)$, so does $\phi$. Thus, if \begin{align*}\overline{\mathcal E}_e(n)& :=\bigcup_{k\geq 0}\overline{\mathcal E}(n,2k)\\ \overline{\mathcal E}_o(n)& :=\bigcup_{k\geq 1}\overline{\mathcal E}(n,2k-1),  \end{align*} then, $|\overline{\mathcal A}_e(n)|-|\overline{\mathcal A}_o(n)|= |\overline{\mathcal E}_e(n)|-|\overline{\mathcal E}_o(n)|$. The proof of Theorem \ref{cor2.2}, shows combinatorially that $|\overline{\mathcal E}_e(n)|=M_e(n)$ and $|\overline{\mathcal E}_o(n)|=M_o(n)$.
\end{proof}

\begin{theorem}\label{thm2.9}\cite[Theorem 2.9]{missing}
    Let $k, n$ be a positive integers. The number of distinct integers less than or equal to the largest part appearing less than $k$ times in all the partitions of $n$ equals the number of parts different from  $k$ in all the partitions of $n$. 
\end{theorem}

\begin{remark} Note that Theorem \ref{thm2.9} is the a complement to Elder's theorem which states that the number of different parts with multiplicity at least $k$ in all the partitions of $n$ equals the number of parts equal to $k$ in all the partitions of $n$. 
In \cite[Exercise 80]{stanley_ec1}, Stanley  gave a bijective  proof of Elder's theorem. Moreover, in the solution to Exercise 80, Stanley also gives a brief insight into the interesting history of the theorem. 
We  give a direct combinatorial proof of Theorem \ref{thm2.9}. 
\end{remark}

\begin{proof}[Proof of Theorem \ref{thm2.9}] 

Let $k, n \geq 1$.  
The integers less than or equal to the largest part $\lambda_1$ of a partition $\lambda$ appearing less than $k$ times  in  $\lambda$ are integers $j\leq \lambda_1$ such that $0\leq m_\lambda(j)<k$, i.e., the missing integers in $\lambda$ and well as the different parts of $\lambda$ with multiplcity less than $k$. Given a partition $\lambda$, we define the \textit{right border}, $B_\lambda$, of $\lambda$ to be the  subdiagram of the Ferrers diagram of $\lambda$ consisting of all boxes with no box to its right, or no box below, or both (no box to its right and below). We write $(i,j)\in B_\lambda$  if $(i,j)\in \lambda$ and $(i, j+1)\not \in \lambda$ or $(i+1, j)\}\not\in \lambda$.  We  note that the right border of a partition is different from the classical definition of a border strip in a Ferrers diagram in which a configuration \ $\scalebox{.4}{\ydiagram{2, 1}}$ \ is allowed as a subdiagram.   Given a partition $\lambda$ of $n$ we decorate its Ferrers diagram by placing symbols in some of the boxes in $B_\lambda$ as follows. 

 $\bullet$ in $(i,j)\in B_\lambda$ if $j\neq k$ and $(i+1, j)\in B_\lambda$; 

 $*$ in $(i,j)\in B_\lambda$ if $i\neq k$ and $(i, j+1)\in B_\lambda$;

 $\a$ in $(k,j)$ if $j\neq \lambda_k-k+1$ and $(k, j+1)\in B_\lambda$; 

 $\b$ in $(k,\lambda_k-k+1)$ if $(k,\lambda_k-k+2)\in B_\lambda$; 

 $s$ in $(i,j)$ if $j\neq k$ and $(i+1, j), (i, j+1), (i-k+1, j)\not\in 
 B_\lambda$;  

 $\ks$ in $(i,k)$ if and $(i+1, k), (i, k+1), (i-k+1, k)\not\in 
 B_\lambda$;

 $\ell$ in $(i,j)$ if $j\neq 1,k$ and $(i+1, j), (i, j+1)\not\in 
 B_\lambda$ and $(i-k+1, j)\in B_\lambda$; 

 $\1$ in $(i,1)$ if $(i+1, 1), (i, 2)\not \in B_\lambda$ and $(i-k+1, 1)\in B_\lambda$.

\vspace{.1in}

\noindent  To summarize, in $B_\lambda$, we have 
 
 $\bullet$ in columns but not in corners and not in the $k$-th column;

 $*$ in rows but not in corners and not in the $k$-th row;

 $\a$ in the $k$-th row but not in the corner or the $k$-th box from the right; 

 $\b$ (if $k>1$) in the $k$-th row and the  $k$-th box from the right; 

 $s$, respectively $\ks$, in corners at the end of ``short" columns of length at most $k-1$;

 $\ell$, respectively $\1$, in corners at the end of ``long" columns of length at least $k$.

 We illustrate a decorated Ferrers diagram in  Example \ref{ex_dec}.

Note that  there is no symbol in the $k$-th column of $B_\lambda$ except possibly $\ks$ in the corner box. If $1\in \lambda$, the last box of the first column in $B_\lambda$ has $s$ or $\1$. If $k=1$, there is no $\b$, $\ks$ or $\1$ in a box in $B_\lambda$.

The number of distinct integers less than or equal to the largest part appearing less than $k$ times in all the partitions of $n$ equals the total number of $*$, $\a$, $\b$, $s$ and $\ks$ in all the partitions of $n$. 
Note that, of these, the total number of $*$, $\a$, $\b$ in all partitions of $n$ equals the number of missing integers in all partitions of $n$. 
The total number of parts different from $k$ in all the partitions of $n$ equals the total number of $\bullet$, $\ell$, $s$ and $\1$ in all the partitions of $n$. 

Conjugations maps $\bullet$ in $\lambda$ to $*$ in $\lambda'$ and vice-versa. 
It remains to show that the total number of  $\a$, $\b$ and $\ks$ in all the partitions of $n$ equals the total number of  $\ell$ and $\1$ in all the partitions of $n$. 
We adapt Stanley's bijection for Elder's theorem \cite{stanley_ec1}. If necessary for transformations to make sense, we add  parts equal to $0$ to a partition. 

We first show that the total number of  $\a$ in all the partitions of $n$ equals the total number of  $\ell$ in all the partitions of $n$. 
If there is a $\a$ in position $(k, j)$,  $j\neq \lambda_k-k+1$, create a partition $\eta$ from $\lambda$ by subtracting $\lambda_k-j+1$ from each of the first $k$ parts of $\lambda$ and inserting $k$ parts equal to $\lambda_k-j+1$ into the obtained partition. 
Then, in the Ferrers diagram of $\eta$, there is an $\ell$ in the last box of the last row of length $\lambda_k-j+1$. For an illustration, see Example \ref{ex_x}.
Conversely, if a partition $\eta$ has $\ell$ in the last box of the last row equal to $i$, create a partition $\lambda$ from $\eta$ by removing $k$ parts equal to $i$ and adding $i$ to each of the first $k$ parts of the obtained partition. 
Then, in $\lambda$, the box in position $(k, \lambda_k-i+1)$ is marked $\a$. 

Next, we show that the total number of $\ks$ and $\b$   in all the partitions of $n$ equals the total number of   $\1$ in all the partitions of $n$.

Suppose  $\lambda$ is partition with $1\leq m_\lambda(k)<k$, i.e., $\lambda$  has $\ks$ in the last row of length $k$. (This can only occur if $k>1$.)
We create a partition $\zeta$ from $\lambda$ by removing a part of length $k$ and inserting $k$ parts equal to $1$. 
Then, the box in the last row equal to $1$ in $\zeta$ has $\1$. 
Moreover, $m_\zeta(k)\leq k-2$. Hence, $\ks$ in $\lambda$ corresponds to $\1$ in $\zeta$.

Now suppose $\lambda$ is a partition  with $\b$ in the box in position $(k, \lambda_k-k+1)$ (This can only occur if $k>1$.) We create a partition $\nu$ from $\lambda$ by subtracting $k$ from each of the first $k$ parts of $\lambda$
and inserting $k-1$ parts equal to $k$ and $k$ parts equal to $1$. In the partition $\nu$ there is $\1$ in the last row of size $1$. Moreover, $m_\nu(k)\geq k-1$. Hence, $\b$ in $\lambda$ corresponds to $\1$ in $\nu$.

These transformations are reversible. 
 \end{proof}


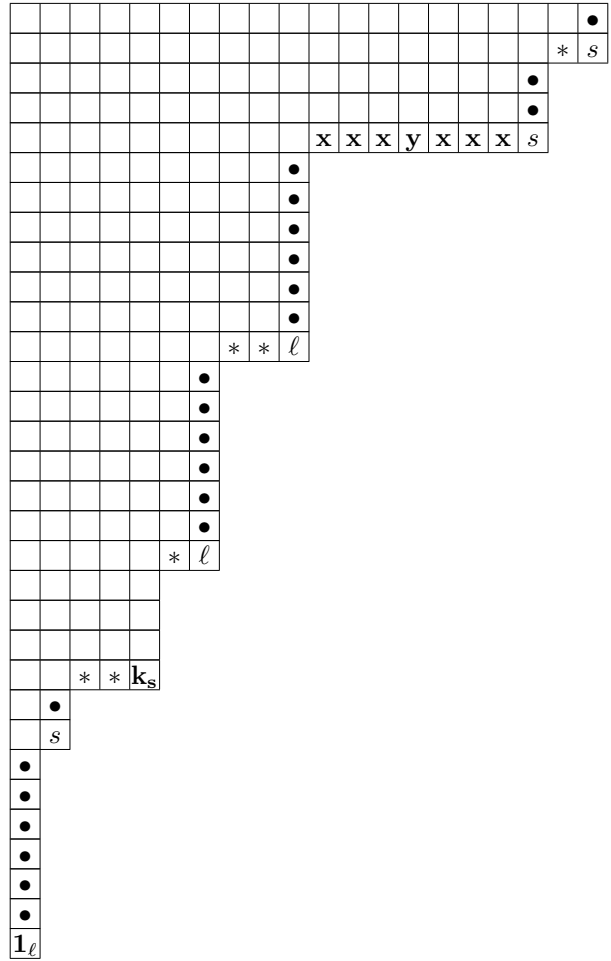
\begin{example} \label{ex_dec}
    If $k=5$ and $\lambda=(20^2, 18^3, 10^7, 7^7, 5^4, 2^2, 1^7)$, the Ferrers diagram of $\lambda$  is decorated as shown below 

$$\young(\hfil\hfil\hfil\hfil\hfil\hfil\hfil\hfil\hfil\hfil\hfil\hfil\hfil\hfil\hfil\hfil\hfil\hfil\hfil\bullet,\hfil\hfil\hfil\hfil\hfil\hfil\hfil\hfil\hfil\hfil\hfil\hfil\hfil\hfil\hfil\hfil\hfil\hfil*s,\hfil\hfil\hfil\hfil\hfil\hfil\hfil\hfil\hfil\hfil\hfil\hfil\hfil\hfil\hfil\hfil\hfil\bullet,\hfil\hfil\hfil\hfil\hfil\hfil\hfil\hfil\hfil\hfil\hfil\hfil\hfil\hfil\hfil\hfil\hfil\bullet,\hfil\hfil\hfil\hfil\hfil\hfil\hfil\hfil\hfil\hfil\a\a\a\b\a\a\a s,\hfil\hfil\hfil\hfil\hfil\hfil\hfil\hfil\hfil\bullet,\hfil\hfil\hfil\hfil\hfil\hfil\hfil\hfil\hfil\bullet,\hfil\hfil\hfil\hfil\hfil\hfil\hfil\hfil\hfil\bullet,\hfil\hfil\hfil\hfil\hfil\hfil\hfil\hfil\hfil\bullet,\hfil\hfil\hfil\hfil\hfil\hfil\hfil\hfil\hfil\bullet,\hfil\hfil\hfil\hfil\hfil\hfil\hfil\hfil\hfil\bullet,\hfil\hfil\hfil\hfil\hfil\hfil\hfil**\ell,\hfil\hfil\hfil\hfil\hfil\hfil\bullet,\hfil\hfil\hfil\hfil\hfil\hfil\bullet,\hfil\hfil\hfil\hfil\hfil\hfil\bullet,\hfil\hfil\hfil\hfil\hfil\hfil\bullet,\hfil\hfil\hfil\hfil\hfil\hfil\bullet,\hfil\hfil\hfil\hfil\hfil\hfil\bullet,\hfil\hfil\hfil\hfil\hfil*\ell,\hfil\hfil\hfil\hfil\hfil,\hfil\hfil\hfil\hfil\hfil,\hfil\hfil\hfil\hfil\hfil,\hfil\hfil**\ks,\hfil\bullet,\hfill s,\bullet,\bullet,\bullet,\bullet,\bullet,\bullet,\1)$$
\end{example}

\begin{example}\label{ex_x} In this example we show only the relevant decoration. Let $k=3$ and $\lambda=(7, 6, 6, 2, 2, 1)$ with box in position $(k,j)=(3,5)$ marked with $\a$. Then, $\lambda_3-j+1=6-5+1=2$. The corresponding partition $\eta$ is obtained from $\lambda$ by subtracting $2$ from each of the first three rows and inserting three rows of size $2$. In $\eta$ the last row of size $2$ is marked with $\ell$ in the last box.

$$\young(\hfil\hfil\hfil\hfil\hfil\hfil\hfil,\hfil\hfil\hfil\hfil\hfil\hfil,\hfil\hfil\hfil\hfil\a\hfil,\hfil\hfil,\hfil\hfil,\hfil)\hspace{2cm} \young(\hfil\hfil\hfil\hfil\hfil,\hfil\hfil\hfil\hfil,\hfil\hfil\hfil\hfil,\hfil\hfil,\hfil\hfil,\hfil\hfil,\hfil\hfil,\hfil\ell,\hfil)$$ $$\lambda \hspace{4.2cm}\eta$$
    
\end{example}

 As noted above, if $k=1$, none of  $s$, $\b$, $\ks$ and $\1$ occur  in the decorated Ferrers diagram of a partition. In this case we obtain a simpler combinatorial proof. 

\begin{corollary}\label{cor2.10}\cite[Corollary 2.10]{missing} 
    The number of missing integers in all the partitions of $n$ equals the number of parts different from $1$ in all the partitions of $n$.
\end{corollary}
\begin{proof}
    If $n=1$, the statement is vacuously true. 
Let $n\geq 2$. 
Given a partition $\lambda$ of $n$, we decorate its Ferrers diagram as in the proof of Theorem \ref{thm2.9}.  

For example, the Ferrers diagram of the partition $\lambda=(13, 9, 9, 5, 5, 5, 3, 1, 1)$ is decorated as shown below 

$$\young(\hfil\hfil\hfil\hfil\hfil\hfil\hfil\hfil\hfil\a\a\a\ell,\hfil\hfil\hfil\hfil\hfil\hfil\hfil\hfil\bullet,\hfil\hfil\hfil\hfil\hfil***\ell,\hfil\hfil\hfil\hfil\bullet,\hfil\hfil\hfil\hfil\bullet,\hfil\hfil\hfil*\ell,\hfil*\ell,\hfil,\hfil)$$

We need to show that the total number of $\bullet$ and $\ell$ in all the partitions of $n$ equals the total number of $*$ and $\a$ in all the  of partitions of $n$. 
Conjugation changes $*$ in $\lambda$ to $\bullet$ in $\lambda'$ and vice-versa. 
Hence, the total number of $\bullet$  in all the partitions of $n$ equals the total number of $*$  in all the  of partitions of $n$. 

To show that the total number of  $\ell$ in all the partitions of $n$ equals the total number of  $\a$ in all the  of partitions of $n$ we proceed as in the proof of Theorem \ref{thm2.9}. 
An $\a$ in position $(1,j)$ in $\lambda$, corresponds to $\ell$ at the end of the last row equal to $\lambda_1-j+1$ in  $\eta$, where $\eta$ is obtained from $\lambda$ by
subtracting $\lambda_1-j+1$ from $\lambda_1$ and inserting a part equal to $\lambda_1-j+1$. 
This transformation is reversible and completes the combinatorial proof of the case $k=1$. 
\end{proof}

Next, we consider results in \cite{missing} concerning overpartitions. Let $\overline{\mathcal P}(n,m)$ be the set of overpartitions of $n$ with exactly $m$ integers less than the largest part missing and define  $\overline{P}(n,m):=|\overline{\mathcal P}(n,m)|$.

\begin{theorem}\cite[Corollary 3.2]{missing} \label{cor3.2}
We have  \begin{equation}\label{eq:cor3.2} \sum_{n=0}^\infty\sum_{m=0}^\infty \overline{P}(n,m)w^mq^n= \frac{((w-2)q;q)_\infty}{(wq;q)_\infty}. \end{equation}
\end{theorem}

\begin{proof}
As in the proof of Theorem \ref{cor2.2}, the coefficient of $q^n$ in the $q$-series expansion of the right-hand side of \eqref{eq:cor3.2} is given by
\begin{equation*}
    \displaystyle\sum_{(\lambda, \mu)\in \mathcal A(n)}\sum^{\ell(\lambda)}_{k=0} (-1)^{k} \binom{\ell(\lambda)}{k} 2^{\ell(\lambda)-k}w^{k+\ell(\mu)}
\end{equation*}

Let $$\overline{\overline{\mathcal A}}(n):=\{(\lambda, \mu)\vdash n \mid \lambda \in \overline{\overline{\mathcal D}}, \mu \in \mathcal P\}$$ and  $$\overline{\overline{\mathcal A}}(n,m):=\{(\lambda, \mu)\in \overline{\overline{\mathcal A}}(n) \mid \overline \ell(\lambda)=m-\ell(\mu)\}.$$

Then, the coefficient of $w^mq^n$ in in the  series expansion of the right hand side of \eqref{eq:cor3.2} is 
$$\#\{(\lambda, \mu)\in \overline{\overline{\mathcal A}}(n,m)\mid \overline\ell(\lambda) \text{ even}\}-\#\{(\lambda, \mu)\in \overline{\overline{\mathcal A}}(n,m)\mid \overline\ell(\lambda) \text{ odd}\}. $$ 

Let $$\overline{\overline{\mathcal E}}(n,m):=\{(\lambda, \mu)\in \overline{\overline{\mathcal A}}(n,m)\mid \ell(\mu)=m \text{ and if }  i\in \mu \text{ then } \mathbf i \in \lambda\}.$$ Note that if $(\lambda,\mu)\in \overline{\overline{\mathcal E}}(n,m)$, then $\overline\ell(\lambda)=0$. 

The involution $\varphi_m$ from the proof of Theorem \ref{cor2.2} can be extended to become an involution $\overline \varphi_m$ on $\overline{\overline{\mathcal A}}(n,m)\setminus \overline{\overline{\mathcal E}}(n,m)$ that reverses the parity of $\overline\ell(\lambda)$. 
 Hence, the coefficient of $w^mq^n$ in  the  series expansion of the right hand side of \eqref{eq:cor3.2} equals $|\overline{\overline{\mathcal E}}(n,m)|$.
We establish a bijection $$\overline \psi_m:\overline{\overline{\mathcal E}}(n,m)\to \overline{\mathcal P}(n,m).$$ Start with $(\lambda, \mu)\in \overline{\overline{\mathcal E}}(n,m)$. 
We create the overpartition $\lambda\cup \mu$ whose overlined parts are exactly the doubly overlined parts of $\lambda$.

We decorate the  Ferrers diagram of the overpartition $\lambda \cup \mu$ as in Theorem \ref{cor2.2}:  for each part $j\in \mu$ we mark  the last box of the first $m_\mu(j)$ rows of length $j$  with $\bullet$. 
The last row of a fixed length is shaded if there is a doubly overlined part of that size in $\lambda$, else the last box  remains unmarked. 
We define $\overline\psi_m(\lambda, \mu):=(\lambda\cup \mu)'$. 
As in the proof of Theorem \ref{cor2.2}, to see that $(\lambda\cup \mu)'\in \overline{\mathcal P}(n,m)$, when conjugating the decorated Ferrers diagram of $\lambda\cup \mu$,  if  the last box of a row is marked  $\bullet$, we mark   the last box of the corresponding column with $*$.    
If there is a $*$ in position $(i,j)$ in the $*$-decorated Ferrers diagram for $(\lambda\cup \mu)'$, then $j\not\in (\lambda\cup \mu)'$. 
Hence, there are exactly $\ell(\mu)=m$ missing integers in $(\lambda\cup \mu)'$ and thus $(\lambda\cup \mu)'\in \overline{\mathcal P}(n,m)$.
\end{proof}

\begin{remark} We note that the proof of 
    Theorem \ref{cor3.2} provides a combinatorial argument for the following identity. For $n,m\geq 0$,
    $$\overline P(n,m) = \#\{(\lambda, \mu)\in \overline{\overline{\mathcal A}}(n,m)\mid \overline\ell(\lambda) \text{ even}\}-\#\{(\lambda, \mu)\in \overline{\overline{\mathcal A}}(n,m)\mid \overline\ell(\lambda) \text{ odd}\}. $$
\end{remark}

\begin{remark} As in the case of ordinary partitions, the cases $m=0$ and $m=1$ reduce to proving \cite[Corollaries 3.3 and 3.4]{missing} via conjugation. 
    \end{remark}

    Finally, we adapt the combinatorial proof of Theorem \ref{cor3.2} to prove \cite[Corollary 3.5]{missing} combinatorially. We first introduce the necessary notation similar to that of \cite{missing}.

    \begin{definition} \cite{missing} We define the sets \begin{align*}\overline{\mathcal M}_e(n)& :=\{\lambda\vdash n \mid \m(\lambda) \text{ even}\}=\bigcup_{k\geq 0}\overline{\mathcal P}(n,2k)\\ \overline{\mathcal M}_o(n)& :=\{\lambda\vdash n \mid \m(\lambda) \text{ odd}\}=\bigcup_{k\geq 1}\overline{\mathcal P}(n,2k-1),\end{align*} and set $\overline M_e(n):=|\overline{\mathcal M}_{e}(n)|$ and $\overline M_o(n):=|\overline{\mathcal M}_{o}(n)|$.
    The functions $\overline{M}_e(n)$ and $\overline{M}_o(n)$ denote the number of overpartitions  of $n$ with an even/odd number of missing integers, respectively. 
\end{definition}

\begin{corollary} \cite[Corollary 3.5]{missing} We have,  \begin{equation}\label{eq:cor3.5} \sum_{n=0}^\infty\sum_{m=0}^\infty \overline P(n,m)(-1)^mq^n= \sum_{n=0}^\infty \left( \overline M_e(n)- \overline M_o(n)\right) q^n=\frac{(-3q;q)_\infty}{(-q;q)_\infty}. \end{equation}
    \end{corollary}
    \begin{proof}
The proof is similar to that of Corollary \ref{cor2.5}. We give the details for the benefit of the reader. First note that  $(-3q;q)_\infty$ is the generating function for the number of distinct doubly overlined partitions of $n$.  Let \begin{align*}
        \overline{\overline{\mathcal A}}_e(n)& :=\{(\lambda, \mu)\in \overline{\overline{\mathcal A}}(n)\mid \ell(\mu) \text{ even}\}\\ \overline{\overline{\mathcal A}}_o(n)& :=\{(\lambda, \mu)\in \overline{\overline{\mathcal A}}(n)\mid \ell(\mu) \text{ odd}\}.
    \end{align*}
Then, $\displaystyle\frac{(-3q;q)_\infty}{(-q;q)_\infty}$ is the generating function for the sequence $|\overline{\overline{\mathcal A}}_e(n)|-|\overline{\overline{\mathcal A}}_o(n)|$. We have $\overline{\overline{\mathcal A}}(n)=\ds\bigcup_{m\geq 0}\overline{\overline{\mathcal A}}(n,m)$ and we define 
 $$\overline{\overline{\mathcal E}}(n):=\bigcup_{m\geq 0}\overline{\overline{\mathcal E}}(n,m)=\{(\lambda, \mu) \in \overline{\overline{\mathcal A}}(n) \mid \overline \ell(\lambda)=0 \text{ and if }i\in \mu \text{ then }\mathbf i \in \lambda\}.$$ From the definitions, it follows that  $\overline{\overline{\mathcal A}}(n)\setminus \overline{\overline{\mathcal E}}(n)=\ds \bigcup_{m\geq 0}\overline{\overline{\mathcal A}}(n,m)\setminus \overline{\overline{\mathcal E}}(n,m)$. We define an involution $\overline\phi$ on $\overline{\overline{\mathcal A}}(n)\setminus \overline{\overline{\mathcal E}}(n)$ as follows: if $(\lambda,\mu)\in \overline{\overline{\mathcal A}}(n)\setminus \overline{\overline{\mathcal E}}(n)$ and $\ell(\mu)+\overline \ell(\lambda)=m$, then $(\lambda, \mu)\in \overline{\overline{\mathcal A}}(n,m)\setminus \overline{\overline{\mathcal E}}(n,m)$ and we set $\overline\phi(\lambda,\mu)=\overline\varphi_m(\lambda,\mu)$. Since $\overline\varphi_m$ changes the parity of $\ell(\mu)$, so does $\overline\phi$. Thus, if \begin{align*}\overline{\overline{\mathcal E}}_e(n)& :=\bigcup_{k\geq 0}\overline{\overline{\mathcal E}}(n,2k)\\ \overline{\overline{\mathcal E}}_o(n)& :=\bigcup_{k\geq 1}\overline{\overline{\mathcal E}}(n,2k-1),  \end{align*} then, $|\overline{\overline{\mathcal A}}_e(n)|-|\overline{\overline{\mathcal A}}_o(n)|= |\overline{\overline{\mathcal E}}_e(n)|-|\overline{\overline{\mathcal E}}_o(n)|$. The proof of Corollary \ref{eq:cor3.2}, shows combinatorially that $|\overline{\overline{\mathcal E}}_e(n)|=\overline M_e(n)$ and $|\overline{\overline{\mathcal E}}_o(n)|=\overline M_o(n)$.
        \end{proof}

\section{Conclusion}\label{conclusion} In  \cite{missing}, Bhoria, Eyyunni, and Santra asked for bijective proofs for several of their results. We were able to provide such proofs for the results involving identities. It is still an open problem to find combinatorial proofs of their congruences as well as an analog for overpartitions of Theorem \ref{thm2.9} \cite[Theorem 2.9]{missing}.

\bibliographystyle{siam}
\bibliography{references}
\end{document}